\documentclass[12pt, twoside, leqno]{article}

\usepackage{amsmath,amsthm}
\usepackage{amssymb}

\usepackage{enumerate}

\usepackage{graphicx}

\newtheorem{thm}{Theorem}[section]

\newtheorem{lem}[thm]{Lemma}

\theoremstyle{definition}
\newtheorem{defin}[thm]{Definition}

\newtheorem{exa}[thm]{Example}

\numberwithin{equation}{section}

\usepackage{pgf,tikz}
\usetikzlibrary{arrows}

\begin{document}

\baselineskip=13pt


\title{Johnson graphs are panconnected}

\author{Akram  Heidari and S.Morteza Mirafzal*\\
Department of Mathematics \\
  Lorestan University, Khorramabad, Iran\\
\\
E-mail: mirafzal.m@lu.ac.ir,\\ 
E-mail:smortezamirafzal@yahoo.com\\E-mail:heidari\_math@yahoo.com}

\date{}

\maketitle


\renewcommand{\thefootnote}{}

\footnote{2010 \emph{Mathematics Subject Classification}: Primary 05C70,   Secondary 05C40,  94C15}

\footnote{\emph{Keywords}: Johnson graph, square of a graph, panconnected graph.}

\footnote{\emph{*Corresponding author.}}

\renewcommand{\thefootnote}{\arabic{footnote}}
\setcounter{footnote}{0}


\begin{abstract}
For any given $n,m \in \mathbb{N}$ with $ m < n $,    the Johnson graph $J(n,m)$ is defined as the graph whose vertex set is $V=\{v\mid v\subseteq  [n]=\{1,...,n\}, |v|=m\}$, where  two vertices $v$,$w$ are adjacent if and only if $|v\cap w|=m-1$.   A graph $G$ of order $n > 2$ is panconnected if for every two vertices $u$ and $v$, there is
a $u$-$v$ path of length $l$ for every integer $l$ with
$d(u,v) \leq l \leq n-1$.  In this paper,  we prove that the Johnson graph  $J(n,m)$ is a panconnected graph.

\end{abstract}
\maketitle
\section{ Introduction}
\noindent

Johnson graphs arise from the association schemes
of the same name. They are defined as follows.\

Given $n,m \in \mathbb{N}$ with $ m<n  $,
the Johnson graph $J(n,m)$ is defined by: \

(1) The vertex set is the set of all subsets of $[n] = \{ 1,2, ..., n\} $ with cardinality exactly $m$;\

(2) Two vertices are adjacent if and only if the symmetric difference of the corresponding
sets is two.\

The Johnson graph $J(n,m)$ is a vertex-transitive graph [7].
It follows from the definition that  for $m = 1$, the Johnson graph $J(n,1)$ is the complete graph
$K_n$. For $m = 2$ the Johnson graph $J(n,2)$ is the line graph of the complete graph on $n$ vertices,
also known as the triangular graph $T(n)$.  If $M$ is an $m$-subset of the set $[n]=\{1,...,n\}$, then the  complementation of subsets $ M \longmapsto  M^c$ induces an isomorphism $J(n,m) \cong    J(n, n-m
)$.  Hence,   in the sequel, we   assume
without loss of generality that  $m \leq \frac{n}{2}  $. \\
A graph $G$ of order $n>2$ is pancyclic if $G$ contains a cycle
of length $l$ for each integer $l$ with $3 \leq l \leq n $. A graph $G$ of order $n > 2$ is panconnected if for every two vertices $u$ and $v$, there is
a $u$-$v$ path of length $l$ for every integer $l$ with
$d(u,v) \leq l \leq n-1$. Note that if a graph $G$ is   panconnected, then $G$ is pancyclic. A graph $G$ of order $n > 2$ is Hamilton-connected if for
any pair of distinct vertices $u$ and $v$, there is a Hamilton $u$-$v$ path, namely, there is
a $u$-$v$ path of length $n-1$. It is clear that if $G$ is a panconnected graph then $G$ is a Hamilton-connected graph. If $n>2,$ then the graph $K_n,$ the complete graph on $n$ vertices,  is a  panconnected graph.  Hence if $m=1$ then the Johnson graph $J(n,m)$ is a   panconnected graph. Alspach [1] proved that the Johnson graph  $J(n,m)$ is a Hamilton-connected graph.  In this paper,  we show that for every  $n,m$  the Johnson graph  $J(n,m)$ is a   panconnected graph, which
generalizes the Alspach's   result.

\section{Preliminaries}
In this paper, a graph $G=(V,E)$ is
considered as a finite undirected simple graph where $V=V(G)$ is the vertex-set
and $E=E(G)$ is the edge-set. For all the terminology and notation
not defined here, we follow $[2,4,5,7]$.\

The
group of all permutations of a set $V$ is denoted by S$ym(V)$  or
just S$ym(n)$ when $ | V | =n $. A $permutation\  group$ $H$ on
$V$ is a subgroup of S$ym(V)$. In this case we say that $H$  $acts$
on $V$. If $G$ is a graph with vertex-set $V$, then we can view
each automorphism of $G$ as a permutation of $V$, and so $Aut(G)$ is a
permutation group. When the group $\Gamma$ acts   on $V$, we say that $\Gamma$ is
$transitive$ (or $\Gamma$ acts $transitively$  on $V$)   if there is just
one orbit. This means that given any two elements $u$ and $v$ of
$V$, there is an element $ \beta $ of  $\Gamma$ such that  $\beta (u)= v
$.

The graph $G$ is called $vertex$-$transitive $  if  $Aut(G)$
acts transitively on $V(G)$. The action of $Aut(G)$ on
$V(G)$ induces an action on $E(G)$, by the rule
$\beta\{x,y\}=\{\beta(x),\beta(y)\}$,  $\beta\in Aut(G)$, and
$G$ is called $edge$-$transitive$ if this action is
transitive.

The $square$ $graph$   $G^{2}$ of a graph $G$  is the graph with vertex set $V(G)$ in which two vertices are adjacent if and only if their distance in $G$ is at most two.\

A vertex cut of the graph $G$ is a subset $U$ of $V$ such that the subgraph
$G-U$ induced by $V-U$ is either trivial or not connected.  The $connectivity$
$\kappa(G)$ of a nontrivial connected graph $G$ is the minimum cardinality of all vertex
cuts   of $G$. If we denote by $\delta(G)$ the minimum degree of $G$, then $\kappa(G) \leq \delta(G)$. A graph $G$ is called $k$-$connected$ (for $k \in \mathbb{N}$) if $|V(G)| > k$ and $G-X$ is connected for every subset $X \subset V(G)  $ with $|X|< k$. It is trivial that if a  positive integer $m$ is such that $m \leq \kappa(G)$, then $G$ is an $m$-connected graph. In the sequel, we need the following facts.

\begin{thm}$[13]$ If a connected graph $G$ is edge-transitive, then $\kappa(G) = \delta(G)$, where $\delta(G)$ is the minimum degree of vertices of $G$.

\end{thm}

\begin{thm}$[3]$
If  $G$ is a $2$-connected graph, then $G^2$ is Hamilton-connected.
\end{thm}

By Theorem 2.2 and [6] we can conclude the following fact.

\begin{thm} The square of a graph $G$ is panconnected whenever $G$ is a $2$-connected graph.
\end{thm}

\section{Main results}
The Boolean lattice $BL_n, n \geq 1$  is the graph whose vertex set is the set of all subsets of $[n]= \{ 1,2,...,n \}$, where two subsets $x$ and $y$ are adjacent if their symmetric difference has precisely one element.  In the graph $BL_n$, the layer $L_m$ is the set of $m$-subsets of $[n]$.  We denote by $ B(n,m)$, the subgraph of $BL_n$ induced by layers $L_m$ and $ L_{m+1} $. Noting that   $ n \choose m$ =$ n \choose n-m$,  we can deduce  that  $ B(n,m)   \cong   B(n,n-m-1). $  Therefore, in the sequel we assume that $ m < \frac {n}{2}  $. Now, we have the following definition.

\begin{defin}
Let $ n \geq 3 $ be an integer and $ [n] = \{1,2,..., n \} $. Let $ m$ be an integer such that
$1\leq m<\frac{n}{2}$. The graph $ B(n,m)$ is a
 graph with the vertex set  $V=\{v \  |  \ v \subset [n] ,  | v |  \in \{ m,m+1  \} \} $ and the
edge set $ E= \{ \{ v , w \} \  | \  v , w \in V , v \subset w $ or $ w \subset v \} $.
\end{defin}
It is easy to see that $B(3,1)$ is $C_6$, the cycle of order 6.
\begin{exa}
According to the Definition 3.1.  Figure 1. shows  $ B(5,1)$    in the  plane.\
\end{exa}\

\definecolor{qqqqff}{rgb}{0.,0.,1.}
\begin{tikzpicture}[line cap=round,line join=round,>=triangle 45,x=.75cm,y=.9cm]
\clip(-4.46,-0.04) rectangle (8.9,4.78);
\draw (-2.,4.)-- (-4.32,1.98);
\draw (-2.,4.)-- (-2.7,1.98);
\draw (-2.,4.)-- (-1.22,2.);
\draw (-2.,4.)-- (0.24,2.);
\draw (0.,4.)-- (-4.32,1.98);
\draw (0.,4.)-- (1.56,2.);
\draw (0.,4.)-- (3.18,2.02);
\draw (0.,4.)-- (4.56,1.96);
\draw (2.,4.)-- (-2.7,1.98);
\draw (2.,4.)-- (1.56,2.);
\draw (2.,4.)-- (6.,2.);
\draw (2.,4.)-- (7.38,1.98);
\draw (4.,4.)-- (-1.22,2.);
\draw (4.,4.)-- (3.18,2.02);
\draw (4.,4.)-- (6.,2.);
\draw (4.,4.)-- (8.54,1.98);
\draw (6.,4.)-- (0.24,2.);
\draw (6.,4.)-- (4.56,1.96);
\draw (6.,4.)-- (7.38,1.98);
\draw (6.,4.)-- (8.54,1.98);
\draw (1.1,0.95) node[anchor=north west] {Figure 1:  B(5,1)};
\begin{scriptsize}
\draw [fill=qqqqff] (-2.,4.) circle (1.5pt);
\draw[color=qqqqff] (-1.84,4.48) node {$ 1 $};
\draw [fill=qqqqff] (0.,4.) circle (1.5pt);
\draw[color=qqqqff] (0.18,4.4) node {$2$};
\draw [fill=qqqqff] (2.,4.) circle (1.5pt);
\draw[color=qqqqff] (2.14,4.4) node {$3$};
\draw [fill=qqqqff] (4.,4.) circle (1.5pt);
\draw[color=qqqqff] (4.18,4.42) node {$4$};
\draw [fill=qqqqff] (6.,4.) circle (1.5pt);
\draw[color=qqqqff] (6.2,4.42) node {$5$};
\draw [fill=qqqqff] (-4.32,1.98) circle (1.5pt);
\draw[color=qqqqff] (-4.34,1.56) node {$12$};
\draw [fill=qqqqff] (-2.7,1.98) circle (1.5pt);
\draw[color=qqqqff] (-2.7,1.64) node {$13$};
\draw [fill=qqqqff] (-1.22,2.) circle (1.5pt);
\draw[color=qqqqff] (-1.28,1.66) node {$14$};
\draw [fill=qqqqff] (0.24,2.) circle (1.5pt);
\draw[color=qqqqff] (0.22,1.68) node {$15$};
\draw [fill=qqqqff] (1.56,2.) circle (1.5pt);
\draw[color=qqqqff] (1.56,1.66) node {$23$};
\draw [fill=qqqqff] (3.18,2.02) circle (1.5pt);
\draw[color=qqqqff] (3.18,1.72) node {$24$};
\draw [fill=qqqqff] (4.56,1.96) circle (1.5pt);
\draw[color=qqqqff] (4.58,1.66) node {$25$};
\draw [fill=qqqqff] (6.,2.) circle (1.5pt);
\draw[color=qqqqff] (5.96,1.66) node {$34$};
\draw [fill=qqqqff] (7.38,1.98) circle (1.5pt);
\draw[color=qqqqff] (7.32,1.64) node {$35$};
\draw [fill=qqqqff] (8.54,1.98) circle (1.5pt);
\draw[color=qqqqff] (8.54,1.66) node {$45$};
\end{scriptsize}
\end{tikzpicture} \
\newline
Note that in Figure 1. $i= \{  i\}, ij= \{ i,j \}$. 
\
\

By  Definition 3.1. it  is clear that  if $v$ is a vertex of $B(n,m)$ of cardinality $m$ (as a subset of [n]), then $deg(v)=n-m$ and if the cardinality  of $v$ is $m+1$, then $deg(v)=m+1$. Now, it is obvious that the graph $ B(n,m) $ is a regular graph if and only if $n=2m+1$. We know that  every vertex-transitive graph is a  regular graph, thus, if $ n\neq 2m+1$, then the graph $B(n,m)$ is not a vertex-transitive graph. Since $m < \frac{n}{2},$ then $\delta(G)=m+1.$
It is clear that the graph $G=B(n,m)$ is a bipartite graph, with $V(G)= P_1 \cup P_2$, where
$$
P_{1}=\{v\ |\  v\subset [n],\; |v|=m+1\},    P_{2}=\{v\ |\  v\subset [n], |v|=m \}. \ \ \ \ \ \
$$
It follows from {}M\"{u}tze [12] that the graph $B(2m+1,m),$  is a Hamiltonian graph [12]. The graph $B(n,m),$ which is defined in [10] for every $n,m$   has some interesting properties [8,9,10,11].  In the sequel, we need the following facts concerning this class of  graphs.
\begin{lem}
The graph $ B(n,m) $ is a connected graph.

\end{lem}

\begin{proof}
The proof is straightforward (see [10]).
\end{proof}

\begin{lem} If $ G =B(n,m) $,
  then $ G $ is edge-transitive. Moreover,  if $n=2m+1$, then $G $ is vertex-transitive.

\end{lem}

\begin{proof}
See [10].
\end{proof}
Note that the number of vertices of the graph $ B(n,m)$ is  $ { n \choose m } + { n \choose m+1 }= { n+1 \choose m+1 }$,   the order of the Johnson graph $J(n+1,m+1)$.
Let $G$ be the graph $B(n,m)$. We assert that $G^2$, the square of the graph $G$,  is in fact  the Johnson graph $J(n+1,m+1)$.

\begin{thm}
If $G$ is  the graph $B(n,m)$, then $G^2 \cong J(n+1,m+1).$

\end{thm}

\begin{proof}
We know that   the vertex-set of  $G^2$ is the  vertex set of $G$ in which two vertices are adjacent if and only if their distance in $G$ is at most two.  Noting that the graph $G=B(n,m)$ is a bipartite graph, with $V(G)= P_1 \cup P_2$, where
$P_{1}=\{v\ |\  v\subset [n],\; |v|=m+1\},    P_{2}=\{v\ |\  v\subset [n], |v|=m\},$  we deduce that if $v$ and $w$ are vertices of $G$ such that $d(v,w)=2$, then $v,w \in P_1$ or $v,w \in P_2$.\\ If $d(v,w)=2$, and  $ v,w \in P_1$, then there is a vertex $u \in P_2$ such that $P: v,u,w  $ is a 2-path in $G$. In other words, $u$ is an $m$-subset of $[n]$ such that $u \subset v$ and $ u \subset w$. Then we must have $|v \cap w|=m. $\\ On the other hand, if $d(v,w)=2$, and  $ v,w \in P_2$, then there is a vertex $u \in P_1$ such that $P: v,u,w  $ is a 2-path in $G$. In other words, $u$ is a $(m+1)$-subset of $[n]$ such that $v \subset u$ and $w \subset u$. Then we must have $|v \cap w|= m-1.$ \

We now consider the Johnson graph $J(n+1,m+1)$, with the vertex set $W=\{ w\  | \
w\subset [n+1]=\{ 1,2,...,n,n+1 \}, |w| = m+1 \}$. Let $W_1= \{ w \ |\  w\in W, n+1 \notin w \}$  and $W_2= \{ w \ | \  w\in W, n+1 \in w \}$.  Note that if $w \in W_2, $ then we have $w=u\cup \{ n+1\}, $ for some $u \subset [n]$ such that $|u|=m$.   We now define the mapping $f: V(G^2) \rightarrow V(J(n+1,m+1))$ by this rule;\\

$$ f(v) =  \begin{cases}
v \ \ \ \ \ \ \ \ \ \ \ \ \ \ \  if \ |v|=m+1 \\  v\cup \{ n+1 \}\      if\ \  |v|=m\\
 \end{cases} $$ \
\ \newline
It is an easy task to show that $f$ is a graph isomorphism.

\end{proof}

\begin{thm}
Let $n,m \in \mathbb{N}$ with $ n\geq 3, \ m \leq \frac{n}{2}$. Then,  the Johnson graph $J(n,m)$ is  a panconnected graph.

\end{thm}

\begin{proof} Note that if $m=1,$
then $J(n,m) = K_n$, the complete graph on $n$ vertices, which is a   panconnected graph. Now  let $m\geq2$. Hence $m-1 \geq 1$.   Let $G=B(n-1,m-1)$ be the graph which is defined in Definition 3.1. Now,  by Lemma 3.3. and Lemma 3.4. $G$ is a connected edge-transitive graph. Hence,  by Theorem 2.1. $\kappa(G)=\delta(G)$. Since $\delta(G)=(m-1)+1=m$, then $G$ is a 2-connected graph.   Therefore,  by Theorem 2.3.  $G^2$  is  a panconnected graph. Now,  since by Theorem 3.5. $G^2 \cong J(n,m)$,  hence the Johnson graph $J(n,m)$ is a panconnected graph. 
\end{proof}


\begin{thebibliography}{}

\bibitem{} Alspach B,    Johnson graphs are Hamilton-connected.   Ars Mathematica Contemparanea  6 (2013),  21-23.

\bibitem{}  Biggs N.L,   Algebraic Graph Theory 1993 (Second edition),   Cambridge Mathematical Library
(Cambridge University Press; Cambridge).

\bibitem{}Chartrand G, Hobbs A.M,   Jung H.A,  Kapoor S.F,    Nash-
Williams J.A, The square of a block is Hamiltonian-connected. J. Combin. Theory
Ser. B 16 (1974) 290-292.

\bibitem{} Diestel  R,  Graph Theory (4th ed.), Springer-Verlage,  Heildelberg (2010).


 \bibitem{}   Dixon J.D,   Mortimer B,   Permutation Groups,  Graduate Texts in Mathematics 1996; 163:   Springer-Verlag, New
York.



\bibitem{} Fleischner H, In the square of graphs, Hamiltonicity and pancyclicity, hamil-
tonian connectedness and panconnectedness are equivalent concept, Monatsh.
Math. 62 (1976) 125-149.



\bibitem{}  Godsil  C,    Royle G,   Algebraic Graph Theory,  2001,  Springer.


\bibitem{} Mirafzal  S.M,   The automorphism group of the bipartite Kneser graph, Proceedings-Mathematical Sciences, (2019), doi.org/10.1007/s12044-019-0477-9.

\bibitem{} Mirafzal  S.M,   A new class of  integral  graphs constructed from the hypercube,     Linear Algebra Appl. 558 (2018) 186-194.

\bibitem{} Mirafzal  S.M,   Cayley properties of the line graphs induced by  consecutive layers of the hypercube,  Arxive: 1711.02701v5,  submitted.



\bibitem{} Mirafzal  S.M,   Zafari  A, Some algebraic properties of bipartite Kneser graphs,  arXiv:1804.04570 [math.GR]  2018,  (to appear in Ars Combinatoria).


\bibitem{}M\"{u}tze T.M,  Su P,   Bipartite Kneser graphs are Hamiltonian,  Combinatorica, Volume 37, Issue 6,  2017,  1206-1219.

\bibitem{} Watkins  M,  Connectivity of transitive graphs,  J. Combin. Theory 1970; 8:
23-29.







\end{thebibliography}
\end{document}